\documentclass[11pt,a4paper]{amsart}
\usepackage[utf8]{inputenc}
\usepackage[english]{babel}
\usepackage{amsmath}
\usepackage{amsfonts}
\usepackage{amssymb}
\usepackage{hyperref}

\newtheorem{theorem}{Theorem}
\theoremstyle{plain}
\newtheorem{corollary}{Corollary}

\newtheorem{proposition}{Proposition}

\numberwithin{equation}{section}
\numberwithin{example}{section}

\title[Sasakian lift \ldots and $\alpha$-Sasakian Ricci solitons ]{Sasakian lift of K\"ahler manifold and 
    $\alpha$-Sasakian Ricci solitons}
  
\author{Piotr Dacko}

\begin{document}
\begin{abstract}
In this paper we provide a  local construction of a Sasakian manifold given  a  K\"ahler manifold. 
Obatined in this way manifold we call Sasakian lift of K\"ahler base. 
Almost contact metric structure is determined by the operation of the lift 
of vector fields - idea similar to lifts in Ehresmann connections.  
We show that Sasakian lift inherits geometry very close to its  K\"ahler 
base. In some sense geometry of the lift is in analogy with geometry of hypersurface
in K\"ahler manifold. There are obtained structure equations between corresponding Levi-Civita 
connections, curvatures and Ricci tensors of the lift and its base. We study lifts 
of symmetries different kind: of complex structure, of K\"hler metric, and K\"ahler structure automorphisms. 
In connection with $\eta$-Ricci solitons we introduce more general class of manifolds 
called twisted $\eta$-Ricci solitons. As we show class of $\alpha$-Sasakian twisted $\eta$-Ricci solitons 
is invariant under naturally defined group of structure deformations.  As corollary it is proved 
  that orbit of Sasakian lift of steady or shrinking Ricci-K\"ahler  soliton  contains 
 $\alpha$-Sasakian Ricci soliton. In case of expanding Ricci-K\"ahler soliton existence of 
 $\alpha$-Sasakina Ricci solition is assured provided expansion coefficient is small enough.
\end{abstract}
\maketitle

\section{Introduction}
The relations between Sasakian and K\"ahler manifolds is now quite well understood.
In case structure is regular, characteristic vector field or Reeb vector field is 
regular, manifold can be viewed as line or circle bundle over K\"ahler manifold, 
where K\"ahler structure is determined by Sasakian structure.  In the paper 
we consider in some sense reverse construction which allow to create Sasakian manifold 
given K\"ahler manifold. Construction is natural. Obtained in this way Sasakian 
manifolds share many properties of K\"ahler manifolds. 
Construction is of Sasakian lift is purely local - 
so strictly speaking we should rather consider our construction in terms of 
germs of structures.  

There is analogy between our construction and idea of Ehresmann connection on 
fiber bundle $\pi:\mathcal P \rightarrow \mathcal B$, $\dim \mathcal B=n$, 
$\dim \mathcal P=n+k$. Ehresmann connection is some $n$-dimensional distribution 
$\mathcal D$ on total space and there is 1-1 operation  operation between vector fields on
base manifold and sections of $\mathcal D$. 

In terms of structure equation here is analogy with theory of hypersurfaces 
in Riemann manifolds. Say $\iota: (\mathcal M,\bar g) \rightarrow (\mathcal{\bar M}, g)$, 
$\iota$ is inclusion, $\bar g=\iota^* g$. The first structure equation relates connection of manifold 
and connection of hypersurface $\nabla_XY = \bar\nabla_XY+h(X,Y)\xi$, $h$ 
being second fundamental form, $\xi$ normal vector field. Now the first structure equation for the lift 
$\pi: (\mathcal M , g) \rightarrow  (\mathcal B, \bar g) $,  $\mathcal M= \mathcal B^L$ is Sasakian lift, $\mathcal B$ 
its K\"ahler base, reads
\begin{equation}
\bar\nabla_{X^L}Y^L  = (\nabla_XY)^L-\Phi(X^L,Y^L)\xi,
\end{equation}
$\xi$ being Reeb vector field. So covariant derivative of Sasakian lift  is determined by the lift of covariant 
derivative of its K\"ahler base.  The other possible point of view of our construction 
is theory of Riemann submersions with 1-dimensional fibers.

Geometry of Sasakian lift is very close to geometry of its base. For example for the Ricci tensor 
$\bar Ric$, of the lift, tensor field $\rho = Ric(X,\phi Y)$ is totally skew-symmetric, ie. some 
2-form, moreover its related to the K\"ahler-Ricci form $\rho$, and K\"ahler form $\omega$ 
\begin{equation}  
 \bar\rho =\pi^*\rho - 2\pi^*\omega,
 \end{equation}
in particular this Sasakian-Ricci form is closed.
  
In geometric terms we study relations between infinitesimal symmetries of complex 
structure, Killing vector fields on K\"ahler manifold and some class of infinitesimal 
symmetries of Sasakian lift, as we call our construction. In particular there is local 
map between inifinitesimal  symmetries of K\"ahler structure and  almost contact 
structure of Sasakian manifold.  

The other results of this kind are relations between holomorphic space forms 
and Sasakian manifolds of constant $\phi$-sectional curvatures, also K\"ahler Einstein 
manifolds and Sasakian $\eta$-Einstein manifolds. Lift of K\"ahler Einstein base 
is Sasakian $\eta$-Einstein manifold. 

We provide relations between curvatures and Ricci tensors  of K\"ahler manifold and its Sasakian lift.
Obtained results allow us to  show that Sasakian lift of holomorphic space form 
is Sasakian manifold of constant $\phi$-sectional curvature, also that 
lift of K\"ahler Einstein manifold is $\eta$-Einstein Sasakian manifold.

One of important subject of this paper is to study properties of Sasakian lift 
of K\"ahler-Ricci soliton. Our main result here is that Sasakian lift satisfies 
what we call  equation of twisted $\eta$-Ricci soliton 
\begin{align*}
 & Ric +\frac{1}{2}\mathcal L_X g = \lambda g +2C_1 \alpha_X\odot\eta +C_2 \eta\otimes\eta, \\
& \alpha_X = \mathcal L_X\eta,
\end{align*}
where $\lambda$, $C_1$, $C_2$ are some constants and $\odot$ denotes symmmetric tensor 
product. 
 In this paper almost contact metric manifold is 
called $\eta$-Ricci soliton if there is vector field $X$ and 
\begin{equation}
Ric +\frac{1}{2}\mathcal L_X g = \lambda g + \mu \eta\otimes \eta,
\end{equation}
thus our definition is more general than that provided in \cite{ChoKimura}, 
where strictly speaking  $\eta$-Ricci soliton is a metric which 
satisfies above equation, where $X=\xi$.  In case of $\alpha$-Sasakian manifold 
Reeb vector field is Killing, therefore manifold is $\eta$-Ricci soliton only if 
it is $\eta$-Einstein manifold. 

Condition for soliton vector field $X=\xi$ is rather restrictive. For 
example for very wide class of manifolds which satisfy $\mathcal L_\xi\eta=0$, 
and   
$d\Phi = 2f\eta\wedge \Phi$, for some local function $f$, the shape of Ricci tensor 
of strict $\eta$-Ricci soliton is completely determined. Namely Ricci tensor 
is necessary of the form, $h=\frac{1}{2}\mathcal L_\xi \phi$,  
\begin{equation}  
Ric(X,Y) = \alpha g(X,Y)+\beta g(X, h\phi Y)+ \gamma \eta(X)\eta(Y),
\end{equation}
where $ \alpha$, $\beta$, $\gamma$ are some functions. In case $h=0$, manifold is $\eta$-Einstein.

For sake of completness paper contains short exposition of geometry of 
class of deformations of Sasakian manifold. These deformations extend 
well-known $\mathcal D$-homotheties. Main result here is that kind 
of deformation, we call them $\mathcal D_{\alpha,\beta}$-homotheties
\begin{equation}
g |_{\mathcal D}\mapsto g' |_{\mathcal D} = \alpha g|_{\mathcal D}, \quad
g |_{\{\xi\}} \mapsto g'|_{\{\xi\}} = \beta^2 g|_{\{\xi\}},
\end{equation}
$\alpha$, $\beta = const. > 0$,  map $\alpha$-Sasakian 
manifold into another $\alpha'$-Sasakian manifold. 

As we will see action of $\mathcal D_{\alpha,\beta}$-homotheties determines three invariant classes 
of twisted $\alpha$-Sasakian $\eta$-Ricci solitons where $C_1 < \frac{1}{2}$, 
$C_1 = \frac{1}{2}$ or $C_1> \frac{1}{2}$.  
Lift of K\"ahler-Ricci soliton provides example of Sasakian twisted $\eta$-Ricci soliton 
with $C_1 < \frac{1}{2}$.  This raises existence question: do exist $\alpha$-Sasakian 
twisted $\eta$-Ricci solitons with $C_1 \geqslant \frac{1}{2}$?  In some cases we have stronger  result: 
there is $\mathcal D_{\alpha,\beta}$-homothety such that  image is $\alpha$-Sasakian 
Ricci soliton.  As we will see this is the case of lift over K\"ahler-Ricci steady or shrinking solitons. 
The lift of expanding K\"ahler Ricci soliton can always be deformed into $\alpha$-Sasakian 
$\eta$-Ricci soliton and into Ricci soliton if expansion coefficient is small enough. 

Used notation can be confused for the reader. Particularly this concerns the how we use the notion of 
$\alpha$-Sasakian manifold. In some parts of paper this term is used in wider sense: 
manifold is called $\alpha$-Sasakian if there is real constant $c > 0$, 
and covariant derivative $\nabla\phi$ satisfies $(\nabla_X\phi)Y=c(g(X,Y)\xi-\eta(Y)X)$. 
However in expressions like $\frac{\beta}{\alpha}$-Sasakian, it is assumed 
that $c=\frac{\beta}{\alpha}$.

\section{Preliminaries}
In this section we will recall some basic facts about almost contact metric manifolds,
and in particular about Sasakian manifolds.
\subsection{Almost contact metric manifolds}

Let $\mathcal{M}$ be a smooth connected odd-dimensional manifold, 
$\dim \mathcal{M} =2n+1 \geqslant 3$. An almost contact metric structure 
on $\mathcal{M}$, is a quadruple of tensor fields $(\phi,\xi,\eta)$, where 
$\phi$ is $(1,1)$-tensor field, $\xi$ a vector field, $\eta$ a 1-form, and 
$g$ a Riemnnian metric, which satisfy \cite{Blair}
\begin{align}
& \phi^2X = -X + \eta(X)\xi, \quad \eta(\xi) = 1, \\
& g(\phi X,\phi Y) = g(X,Y)-\eta(X)\eta(Y),
\end{align}
where $X$, $Y$ are arbitrary vector fields on $\mathcal{M}$. Triple 
$(\phi,\xi,\eta)$ is called an almost contact structure (on $\mathcal{M}$). 
From definition it follows that tensor field $\Phi(X,Y)=g(X,\phi Y)$, is 
totally skew-symmetric, a 2-form on $\mathcal{M}$, called a fundamental form.
In the literature vector field $\xi$ is referred to as characteristic vector field,
or Reeb vector field. In analogy the form $\eta$ is called characteristic form. 
Distribution $\mathcal{D} = \{\eta =0\}$, is called characteristic distribution,
or simply kernel distribution, as its sections 
$X\in \Gamma^\infty(\mathcal{D})$, satisfy $\eta(X)=0$. As $\eta$ is non-zero everywhere 
$\dim \mathcal{D} = 2n$.  Manifold equipped with fixed almost contact metric 
structure is called almost contact metric manifold.  

Let for $(1,1)$-tensor field $S$, $N_S$ denote its Nijenhuis torsion, thus
\begin{equation}
N_S(X,Y)=S^2[X,Y]+[SX,SY]-S([SX,Y]+[X,SY]).
\end{equation}
Almost contact metric structure $(\phi,\xi,\eta,g)$ is said to be normal 
if tensor field $N^{(1)} = N_\phi +2d\eta\otimes \xi$, vanishes everywhere. 
Normality is the condition of integrability of naturally defined complex structure 
$J$ on a product of real line (a circle) and almost contact metric manifold.

There are three classes most widely studied almost contact metric manifolds
\begin{itemize}
\item[a)] Contact metric manifolds defined by condition 
\begin{equation}
     d\eta = \Phi;
\end{equation}
\item[b)] Almost Kenmotsu manifolds
\begin{equation}
d\eta = 0,\quad d\Phi = 2\eta\wedge\Phi;
\end{equation}
\item[c)] Almost cosymplectic (or almost coK\"ahler) manifolds
\begin{equation}
d\eta =0, \quad d\Phi =0. 
\end{equation}
\end{itemize}
If additionaly manifold is normal we obtain following corresponding classes
\begin{itemize}
\item[an)] Sasakian manifolds - ie. contact metric and normal
\begin{equation}
(\nabla_X\phi)Y = g(X,Y)\xi - \eta(Y)X;
\end{equation}
\item[bn)] Kenmotsu manifolds 
\begin{equation}
(\nabla_X\phi)Y = g(\phi X,Y)\xi - \eta(Y)\phi X;
\end{equation}
\item[cn)] Cosymplectic manifolds
\begin{equation}
\nabla\phi =0.
\end{equation}
\end{itemize}
Above we have provided characterization of respective manifold in terms of 
covariant derivative - with resp. to the Levi-Civita connection of $g$ - 
of the structure tensor $\nabla\phi$.

For an almost contact metric manifold $\mathcal{D}$-homothety, 
with coefficient $\alpha$, is a 
deformation of an almost contact metric structure 
$(\phi,\xi,\eta,g) \rightarrow (\phi',\xi',\eta',g')$, defined by 
\begin{align*}
& \phi'=\phi,\quad \xi'=\frac{1}{\alpha}\xi,\quad 
\eta' = \alpha\eta, \\
& g' = \alpha g +(\alpha^2-\alpha)\eta\otimes\eta.
\end{align*} 
In this paper we will consider more general deformations of an 
almost contact metric structure, defined by real parameters 
$\alpha$, $\beta > 0$, given by 
\begin{align*}
& \phi' = \phi, \quad \xi'=\frac{1}{\beta}\xi, \quad
\eta' = \beta\eta, \\
& g'= \alpha g + (\beta^2-\alpha)\eta\otimes\eta.
\end{align*}
We call such deformations $\mathcal D_{\alpha,\beta}$-homotheties. 

\subsection{Sasakian manifolds}
Here we provide some very basic properties of Sasakian 
manifold. On Sasakian manifold 
\begin{equation}
(\nabla_X\phi)Y = g(X,Y)\xi - \eta(Y)X,
\end{equation}
which implies that  Reeb vector field $\xi$,
is  Killing vector field $\mathcal{L}_\xi g=0$, moreover 
\begin{equation} 
\mathcal{L}_\xi\phi =0, 
\quad \nabla_X\xi = -\phi X,\quad \nabla_\xi\xi=0,
\end{equation}
for curvature $R_{XY}\xi$, and Ricci tensor $Ric(X,\xi)$,
we have
\begin{align}
\label{l:curv2}
& R_{XY}\xi = \eta(Y)X-\eta(X)Y, \quad
R_{X\xi}\xi = X-\eta(X)\xi, \\
& \label{l:ric:xi}
Ric(X,\xi) = 2n\eta(X),
\end{align}
in particular $Ric(\xi,\xi)=2n$, and $R_{XY}\xi =0$, 
for $X$, $Y$  sections of characteristic distribution,
$X$, $Y\in \Gamma^\infty(\mathcal{D})$, see \cite{Blair}.

More general almost contact metric manifold is called $\alpha$-Sasakian  
\begin{equation}
(\nabla_X\phi)Y = \alpha(g(X,Y)\xi -\eta(Y)X),
\end{equation}
for some non-zero real constant $\alpha \neq 0$.

As we will proceed further we will obtain following equation 
for Ricci tensor $Ric$ of $\alpha$-Sasakian manifold
\begin{equation}
\label{e:e:sol}
Ric +\frac{1}{2}(\mathcal{L}_Xg) = \lambda g +2C_1\alpha_X\odot\eta +
C_2\eta\otimes \eta,
\end{equation}
where $\lambda$, $C_1$, $C_2$ are some real constants, vector field
$X$, and 1-form $\alpha_X$, satisfy
\begin{equation}
\label{e:e:sol2}
\eta(X) = 0, \quad (\mathcal{L}_X\eta)(Y)=\alpha_X(Y),
\end{equation}
in particular on $\alpha$-Sasakian manifold $\alpha_X(\xi)=0$, as
by assumption $\eta(X)=0$, hence
$(\mathcal{L}_X\eta)(\xi)=2d\eta(X,\xi)=0$.

\subsection{K\"ahler manifolds}
Almost complex structure on manifold is a $(1,1)$-tensor
 field $J$, such that $J^2X = -Id$. Structure is said 
 to be complex if any point admits a local chart, 
 such that local coefficient of $J$, in this chart, 
 are all constants. Necessary and sufficient condition,
 is vanishing Nijenhuis torsion of $J$. If additionally 
 there is Riemannian metric $g$, with properties 
 $g(JX,JY)= g(X,Y)$, and complex structure is covariant 
 constant for Levi-Civita connection - manifold is called a K\"ahler manifold. Tensor field 
 $\omega(X,Y)=g(X,JY)$, is a maximal rank 2-form, called K\"ahler 
 form, as $J$ is parallel, K\"ahler form is always 
 closed. In particular K\"ahler form determines symplectic structure on K\"ahler manifold.  
 
 In dimensions $ > 2$,  K\"ahler manifold of constant 
sectional curvature is always locally flat, thus more 
natural notion is holomorphic sectional curvature. 
This is sectional curvature of "complex" plane 
ie. a plane spanned by vectors $X$, and $JX$. If 
holomorphic curvature does not depend neither on 
point nor on complex plane section - K\"ahler 
manifold is said to have constant holomorphic 
curvature. 

Infinitesimal automorphism of complex structure is 
a vector field $X$, which generates local 1-parameter 
flow $f_t$, of biholomorphisms, that is 
$f_{t*}J = Jf_{t*}$. The vector field $X$ is an 
infinitesimal automorphism if and only if 
$\mathcal{L}_XJ =0$. Note that if $X$ determines 
infinitesimal automorphism, then also vector 
field $JX$ is an infinitesimal automorphism.
Moreover complex vector field 
$X^\mathbb{C}=X-\sqrt{-1}JX$, 
is holomorphic. Its local coordinates in complex chart are holomorphic functions. 

We say that vector field $X$ is K\"ahler structure automorphism if 
the two of the three conditions 
\begin{equation}
\mathcal{L}_XJ = 0,\quad \mathcal{L}_Xg =0, \quad \mathcal{L}_X\omega =0,
\end{equation}
are satisfied. Then the third condition is automatically satisfied. For 
example if $X$ is complex structure automorphism and Killing vector field, then
it preserves K\"aler form. In particular is locally Hamiltonian with respect 
to symplectic structure determined by K\"ahler form, that is locally 
\begin{equation}
   \omega(X,Y)= dH(Y),
\end{equation}
for some locally defined function $H$. cf. \cite{Bess},\cite{KoNo2}.

\subsection{Ricci solitons}
Given geometric objects on manifold is it important to know 
do exist objects with some particular properties. One of 
possible way to find such particular entity is trough geometric 
flow. Of course here the point is how to create proper law of evolution. 
This kind of object is so-called Ricci flow introduced by Hamilton 
\cite{Hamilton}
\begin{equation}
\frac{\partial}{\partial t}g=  -2Ric(g_t), \quad t\in [0,T), \quad T > 0, 
\end{equation}
where we search for solution on some non-empty interval, with given 
initial condition $g_0 = g$.  

In present time Ricci flow is one of most extensively studied subject. The 
goals are two-fold: analytical and geometrical. In terms of analysis 
there are studies considering problems of existence of the Ricci, on 
side of geometry plenty new manifolds which admits non-trivially Ricci 
flows \cite{ChowKnopf},\cite{ChowLuNi}. 

Particular case of solutions are flows of the form 
\begin{equation}
g_t = c(t)f_t^*g, 
\end{equation}
where $f_t$  $1$-parameter group of diffeomorphisms. Such solutions are called Ricci solitons. In some sense they represent 
trivial solutions of Ricci flow, say $g_0$, and $g_t$,  
are always  isometric up to homothety. In infinitesimal terms 
Ricci soliton is Riemannian manifold (Riemannian metric),
 which admits a vector field 
$X$, such that 
\begin{equation}
Ric +\frac{1}{2}\mathcal{L}_Xg = \lambda g, 
\end{equation}
for some real constant $\lambda \in \mathbb{R}$, \cite{Cao1},\cite{Cao2},\cite{Ivey},\cite{SongWeink}. 
 Depending on sign of $\lambda$ there are expanding Ricci solitons: $\lambda >0$,  
 steady: $\lambda =0$, and shrinking: $\lambda < 0$.  Of course in this classification only sign  
 of $\lambda$ counts as homothety $g \mapsto g' = cg$, $c > 0$, provides Ricci soliton with 
 soliton constant $\lambda' = \lambda c$ (and soliton vector field $X' = cX$). 
 In particular we can always normalize equation to $\lambda =1, 0, -1$.
 Assuming $X$ is gradient 
$g(X,Y)=dH(Y)$, we have $(\mathcal{L}_Xg)(Y,Z) = 2Hess_H(Y,Z)$, where 
as usually $Hess_H$, stands for Hessian of the function, which is defined by 
\begin{equation}
Hess_H(Y,Z)=(\nabla_Y dH)(Z),
\end{equation} 
ie. it is covariant derivative of differential form of the function $H$. 
Therefore in case $X=grad H$, we obtain
\begin{equation}
Ric + Hess_H = \lambda g.
\end{equation}

The solution to the Ricci flow equation in case of initial metric being K\"aheler, is  family 
of K\"ahler metrics.  Therefore in particular case of  Ricci soliton we obtain that the 
vector field $X$, satisfies $\mathcal L_X J=0$.   
Equivalently metric of K\"ahler manifold is K\"ahler-Ricci soliton if 
\begin{equation}
\rho +\frac{1}{2}\mathcal{L}_X\omega = \lambda\omega, \quad \mathcal L_X J=0,
\end{equation}
where $\rho$ denotes the Ricci form \cite{Bryant}.

\section{Sasakian lift of K\"ahler manifold}
Let $\mathcal{N}$ be K\"ahler manifold with K\"ahler structure $(J,g)$, let 
$\omega$ be a K\"ahler form $\omega(X,Y) = g(X,JY)$. Assume there is globally 
defined 1-form $\tau$, such that $\omega = d\tau$, if there is no such form, 
so real cohomology class $[\omega] \neq 0$, we restrict structure to some 
open subset $\subset \mathcal{N}$, to assure existence of $\tau$. Then on 
product $\mathcal{M} = \mathbb{R}\times\mathcal{N}$, we introduce structure 
of Sasakian manifold in terms of a lift of vector field on $\mathcal{N}$. 
Set 
$$ 
\xi = \partial_t,\quad \eta = dt + \pi_2^*\tau,
$$ 
 for vector 
field $X$ on K\"ahler base we define its lift  $X \mapsto X^L$
$$X^L= -\pi_2^*\tau(X)\xi+X,
$$ where 
$\pi_1$, $\pi_2$ are projections 
$$
\pi_1: \mathcal{M} \rightarrow \mathbb{R},\quad 
 \pi_2: \mathcal{M}\rightarrow \mathcal{N},
 $$ 
 on the first and second 
 product components. If $f:\mathcal{N} \rightarrow \mathbb{R}$, 
 is a smooth function on K\"ahler base, 
we set $\bar{f}:\mathcal{M}\rightarrow \mathbb{R}$, $\bar{f}=f\circ\pi_2$, 
then $(fX)^L = \bar{f}X^L$. Note for $\bar{f}$, we have $\xi\bar{f}=
d\bar{f}(\xi)=df(\pi_{2*}\xi)=0$.  For functions $f_i$, and vector fields $X_i$,  $i=1,2$, on $\mathcal{N}$ we have
\begin{equation}
(f_1 X_1+f_2X_2)^L = \bar{f}_1X_1^L+\bar{f}_2X_2^L,
\end{equation}
in particular $(\cdot)^L$ is $\mathbb{R}$-linear.
  With help of the operation of the lift  we now introduce on $\mathcal{M}$ 
 a tensor field $\phi$, and metric $g^L$ , requiring 
\begin{equation}
\phi X^L = (JX)^L, \quad \phi\xi =0, \quad g^L = \eta\otimes\eta + \pi_2^*g.
\end{equation}
   Note following  formula for 
 commutator of lifts
 \begin{equation}
 \label{e:bra:lif}
  [X^L,Y^L]  = [X,Y]^L-2d\eta(X^L,Y^L)\xi.
 \end{equation}
 
 The following proposition 
 is just simple verification with help of provided definitions.
 \begin{proposition}
 Tensor fields $(\phi^L,\xi,\eta, g^L)$, determine almost contact 
 metric structure on $\mathcal{M}$. 
 \end{proposition} 
 \begin{proof}
 If $(X_1,X_2,\ldots, X_{2n})$ is a local frame of vector fields on K\"ahler base vector 
 fields $(X_1^L,X_2^L,\ldots,X^L_{2n})$ is a local frame which spans characteristic 
 distribution $\mathcal D =  \ker \eta = \{ \eta = 0 \}$, with Reeb vector field they create local 
 frame on the lift $\mathcal M$. Therefore any vector $\bar Y$ on the lift field can be given locally by
 \begin{equation}
\bar Y =a^0\xi + \sum_{i=1}^{2n}a^i X_i^L, 
\end{equation}
$a^i$, $i=0,\ldots, 2n$ are some functions, 
therefore it is enough to verify that $\phi^2 X^L = -X^L$. By definition 
\begin{equation}
\phi^2 X^L = \phi (JX)^L = (J^2 X)^L = -X^L.
\end{equation}
Similarly in case of metric given vector fields $\bar Y$, $ \bar Z$, 
\begin{equation}
g^L(\phi\bar Y,\phi \bar Z) = \sum_{i,j=1}^{2n} a^i b^j g^L(\phi X_i^L,\phi X_j^L),
\end{equation}
from definition we have 
\begin{equation}
g^L(\phi X_i^L,\phi X_j^L)  = g(JX_i,JX_j)\circ \pi_2,
\end{equation}
as base manifold is K\"ahler $g(JX_i,JX_j) = g(X_i,X_j)$, from other hand 
\begin{equation}
g^L(X_i^L, X_j^L ) = g(X_i,X_j)\circ\pi_2,
\end{equation}
hence $g^L(\phi  \bar Y, \phi \bar Z) = g^L(\bar Y, \bar Z)-\eta(\bar Y)\eta(\bar Z)$.
 \end{proof}
  
 The proof o the above Proposition says little more.
 \begin{corollary}
 For $(t,q)\in \mathbb{R}\times \mathcal{N}$,  map $X \mapsto X^L$, establishes 
 isometry $(T_q\mathcal{N}, g_q)\rightarrow (\mathcal{D}_{(t,q)}, g^L|_\mathcal{D})$. 
 In particular if $E_i$, $i=1,\ldots 2n$, is an orthonormal local frame 
 on K\"ahler base, lifts $E_i^L$, $i=1,\ldots 2n$, form an orthonormal frame 
 spanning contact distribution. 
 \end{corollary}

In future we we simplify notation and instead of eg.  
\begin{equation*}
g^L(X^L,Y^L) = g(X,Y)\circ \pi_2,
\end{equation*} 
we write 
$
g^L(X^L,Y^L) = g(X,Y),
$
if it does not lead to a confusion. 
Similarly $d\eta(X^L,Y^L)=\Phi(X^L,Y^L)=\omega(X,Y)$. 
For function on K\"ahler base, 
or according to our simplified  notation $X^L\bar{f}=\overline{Xf}$.

\begin{proposition} 
Manifold $\mathcal{M}$, equipped with structure $(\phi^L,\xi,\eta,g^L)$, 
is Sasakian manifold.
\end{proposition}
\begin{proof}
The first we note that the fundamental form of $\mathcal{M}$ is just 
pullback of K\"ahler form $\omega$, $\Phi = \pi_2^*\omega$. From other 
hand $d\eta = d\pi_2^*\tau = \pi_2^*d\tau = \pi_2^*\omega = \Phi$, by assumption 
about $\tau$, and our above remark. Therefore $\mathcal{M}$ is contact 
metric manifold. To end the proof we directly verify that $\mathcal{M}$, is 
normal $N^{(1)} = 0$, it is enough to verify normality on vector fields 
of form $X^L$, $N^{(1)}(X^L,Y^L)= 0$, as they span the module 
$\Gamma^\infty(\mathcal{D})$, of all sections of contact distribution, and 
to verify directly that $N^{(1)}(\xi,X^L)=0$. As $N^{(1)} = N_\phi +2d\eta\otimes\xi$, 
with help of (\ref{e:bra:lif}) we obtain
\begin{equation}
N_\phi(X^L,Y^L) = (N_J(X,Y))^L -2 d\eta(X^L,Y^L)\xi = -2 d\eta(X^L,Y^L)\xi,
\end{equation}
as $J$ is complex structure, by Newlander-Nirenberg theorem this equivalent to vanishing its Nijenhuis 
torsion $N_J=0$. So 
$$
N^{(1)}(X^L,Y^L) = -2d\eta(X^L,Y^L)\xi +2d\eta(X^L,Y^L)\xi =0.
$$  
The case $N^{(1)}(\xi, X^L)$ is almost evident as for every vector field on K\"ahler base there is 
$[\xi, X^L]=0$, $d\eta(\xi, \cdot) =0$.
\end{proof}

The almost contact metric structure constructed as above  we call Sasakian lift of 
a K\"ahler structure. 
Consequently  manifold itself we call  Sasakian lift of K\"ahler manifold. If it 
is not explicitely stated what particular manifold, we just use a term Sasakian 
lift to emphasize that almost contact metric structure is obtained from 
some K\"ahler base manifold with help of the above described construction.

\subsection{Structure equations}
Here we provide fundamental relations between Levi-Civita connections 
of K\"ahler base and its Saskian lift. Let $\bar\nabla$ denote the 
operator of the covariant derivative of Levi-Civita connection 
of Sasakian lift metric $\bar{\nabla} = LC(g^L)$, while 
$\nabla=LC(g)$, the Levi-Civita connection of  K\"ahler base.

\begin{proposition}
\label{p:streqs}
For vector fields $X$, $Y$ on  K\"ahler base we have
\begin{align}
\label{streqs1}
& \bar{\nabla}_{X^L}\xi = - \phi X^L = -(JX)^L, \\
\label{streqs2}
& \bar{\nabla}_{X^L}Y^L = (\nabla_XY)^L -\Phi(X^L,Y^L)\xi, 
\end{align}
\end{proposition} 
\begin{proof}
The first structure equation comes from property of any Sasakian 
manifold and the definition of the lift. 
The second structure equation is consequence of  Koszul formula 
for Levi-Civita connection, applied to both  K\"ahler 
base and its lift. Using Koszul formula we need to take 
into account that
\begin{equation*}
g^L(X^L,Y^L) = g(X,Y),\quad X^Lg^L(Y^L,Z^L)=Xg(Y,Z).
\end{equation*}
Therefore 
\begin{align*}
& 2g^L(\bar \nabla_{X^L}Y^L, Z^L) = X^Lg^L(Y^L,Z^L)+Y^Lg^L(X^L,Z^L) - \\ 
& \qquad Z^Lg^L(X^L,Y^L) + g^L([X^L, Y^L], Z^L) + g^L([Z^L,X^L],Y^L) + \\
& \qquad g^L([Z^L,Y^L],X^L) = 2 g(\nabla_XY, Z)\circ\pi_2 = 2g^L((\nabla_XY)^L, Z^L),
\end{align*}
from other hand projection  $\bar\nabla_{X^L}Y^L$ on $\xi$, is given by 
$g^L(\bar\nabla_{X^L}Y^L,\xi) = -g^L(\bar{\nabla}_{X^L}\xi,Y^L) = 
-\Phi(X^L,Y^L)\xi$. Here we use only the fact that on Sasakian manifold always 
$\nabla\xi = -\phi$. 
\end{proof}

Note above formula coincides with formula of commutator of the lifts
\begin{align*}
& [X^L,Y^L] = \bar\nabla_{X^L}Y^L-\bar\nabla_{Y^L}X^L = \\ 
& \qquad (\nabla_XY)^L-\Phi(X^L,Y^L)\xi - (\nabla_YX)^L +\Phi(Y^L,X^L)\xi = \\
& \qquad [X,Y]^L -2\Phi(X^L,Y^L)\xi,  
\end{align*}

In words orthogonal projection of covariant derivative of lifts $\bar\nabla_X^LY^L$, on characteristic 
distribution  is equal exactly to the lift of covariant derivative on K\"ahler base $(\nabla_XY)^L$, 
while projection on direction of Reeb vector field is equal to $-\Phi(X^L,Y^L)\xi$,  however 
note that $\Phi(X^L,Y^L) = \omega(X,Y)\circ \pi_2$, ie. pullback of K\"ahler form on these vector fileds. 
In symbolic terms we can describe this as 
\begin{equation}
\bar \nabla = \nabla^L -(\pi_2^*\omega)\otimes\xi. 
\end{equation}
 
The structure equations in the Proposition {\bf \ref{p:streqs}.} 
remind structure equations 
for hypersurface in Riemannian manifold. However there is remarkable difference:
in case of hypersurface its second fundamental form is symmetric tensor field, 
while in our case the tensor which supposedly plays a role of second fundamental 
form is skew-symmetric. Note that (\ref{streqs2}) {\it is not } a definition of connection. 
$\bar \nabla $ is just Levi-Civita connection of the metric $g^L$. But in particular 
case of lifts of vector fields from K\"ahler base, (\ref{streqs2}) holds true.  

Having  the structure equations as above we proceed to obtain relations between corresponding curvature tensors 
of K\"ahler manifold and its Sasakian lift. 

 
\begin{proposition}
\label{p:curv}
Curvatures $R$ and $\bar{R}$ of K\"ahler base and its Sasakian lift are related by
\begin{align}
\label{l:curv1}
& \bar{R}_{X^LY^L}Z^L  =  (R_{XY}Z)^L  + \Phi(Y^L,Z^L)\phi X^L - 
\Phi(X^L,Z^L)\phi Y^L - \\ 
&\nonumber \qquad 2\Phi(X^L,Y^L)\phi Z^L, 
\end{align}
\end{proposition}
\begin{proof}
By the structure equations (\ref{streqs1}), (\ref{streqs2})
\begin{align}
\label{eq1}
& \bar{\nabla}_{X^L}\bar{\nabla}_{Y^L}Z^L = \bar{\nabla}_{X^L}(\nabla_YZ)^L -
X^L\Phi(Y^L,Z^L)\xi + \\
& \nonumber \qquad \Phi(Y^L,Z^L)\phi X^L = (\nabla_X\nabla_YZ)^L - 
(\Phi(X^L,(\nabla_YZ)^L)+ \\
& \nonumber \qquad X^L\Phi(Y^L,Z^L))\xi + \Phi(Y^L,Z^L)\phi X^L, \\
\label{eq2}
& \bar{\nabla}_{[X^L,Y^L]}Z^L = \bar{\nabla}_{[X,Y]^L}Z^L -
2d\eta(X^L,Y^L)\bar{\nabla}_\xi Z^L = \\
& \nonumber \qquad (\nabla_{[X,Y]}Z)^L - \Phi([X,Y]^L,Z^L)\xi - 
2d\eta(X^L,Y^L)\bar{\nabla}_\xi Z^L.
\end{align}
For the lift of vector field $Z^L$, $[\xi, Z^L] = 0$. Therefore as Levi-Civita connection has no torsion, we have  
\begin{equation}
\label{eq3}
\bar{\nabla}_\xi Z^L = \bar{\nabla}_{Z^L}\xi = -\phi Z^L.
\end{equation}  
For curvature  
$$ \bar{R}_{X^LY^L}Z^L = \bar{\nabla}_{X^L}\bar{\nabla}_{Y^L}Z^L -
\bar{\nabla}_{Y^L}\bar{\nabla}_{X^L}Z^L - 
\bar{\nabla}_{[X^L,Y^L]}Z^L,
$$
in view of (\ref{eq1})-(\ref{eq3}),  we obtain 
\begin{align*}
& \bar{R}_{X^LY^L}Z^L = (R_{XY}Z)^L +\Phi(Y^L,Z^L)\phi X^L - 
\Phi(X^L,Z^L)\phi Y^L + \\  
& \nonumber \qquad ( ( -\bar\nabla_{X^L}\Phi)(Y^L,Z^L)  + (\bar\nabla_{Y^L}\Phi)(X^L,Z^L))\xi -
 2\Phi(X^L,Y^L)\phi Z^L , 
\end{align*}
as manifold is Sasakian 
$(\bar{\nabla}_{X^L}\Phi)(Y^L,Z^L)=(\bar{\nabla}_{Y^L}\Phi)(X^L,Z^L)=0$, vanish.
\end{proof}
 The following proposition describes relation between Ricci tensors of K\"ahler base and 
its lift
\begin{proposition}
\label{p:ric}
Ricci tensors and scalar curvatures of K\"ahler base and its Sasakian lift 
are related by
\begin{align}
\label{r1}
& \bar{R}ic(X^L,Y^L) = Ric(X,Y)-2g(X,Y), 
\end{align}
in particular we have for scalar curvatures of K\"ahler base and 
Sasakian lift $s$, $\bar{s}$ 
\begin{align}
\label{r2}
& \bar{s}= s-2n.
\end{align}
\end{proposition}
\begin{proof}
In the proof we use adopted local orthonormal frame $(\xi, E_1^L,\ldots E_{2n}^L)$,  
where $(E_1,\ldots E_{2n})$, is local orthonormal frame on K\"ahler base. Then
\begin{equation}
\label{ric1}
\bar{R}ic(X^L,Y^L) = \bar{R}(\xi, X^L,Y^L,\xi) + 
\sum\limits_{i=1}^{2n}\bar{R}(E_i^L,X^L,Y^L,E_i^L),
\end{equation}  
by the Proposition {\bf \ref{p:curv}.}, (\ref{l:curv1}), and 
curvature identities for Sasakian manifold (\ref{l:curv2}), 
we obtain 
\begin{align}
\label{tr:curv}
& \sum\limits_{i=1}^{2n}\bar{R}(E_i^L,X^L,Y^L,E_i^L) = 
\sum\limits_{i=1}^{2n}R(E_i,X,Y,E_i) - \\
& \nonumber \qquad 3\sum\limits_{i=1}^{2n}\Phi(E_i^L,X^L)\Phi(E_i^L,Y^L) = 
Ric(X,Y)-3g(JX,JY)= \\ 
& \nonumber \qquad Ric(X,Y)-3g(X,Y), \\
\label{xi-sect}
& \bar{R}(\xi,X^L,Y^L,\xi) = \bar{R}(X^L,\xi,\xi,Y^L) =g^L(X^L,Y^L)=g(X,Y),
\end{align} 
now with help of (\ref{ric1})-(\ref{xi-sect}), we find
\begin{equation*}
\bar{R}ic(X^L,Y^L)=Ric(X,Y)-2g(X,Y).
\end{equation*}
For the scalar curvature of the lift 
$
\bar{s}= \bar{R}ic(\xi,\xi)+\sum_{i=1}^{2n}\bar{R}ic(E_i^L,E_i^L),
$
and by (\ref{l:ric:xi}), (\ref{r1}), 
\begin{equation*}
\bar{s}=2n+\sum\limits_{i=1}^{2n}(Ric(E_i,E_i)-2g(E_i,E_i)) =2n +s -4n = s-2n.
\end{equation*}
\end{proof}

\begin{proposition}
On Sasakian lift tensor field $\bar{\rho}(\cdot,\cdot)=\bar{R}ic(\cdot,\phi \cdot)$, 
is a closed 2-form, moreover
\begin{equation}
\bar{\rho} = \pi_2^*\rho -2\pi_2^*\omega,
\end{equation}
that is $\bar{\rho}$ is a pullback of difference of Ricci form and twice 
of K\"ahler 
form.
\end{proposition}
\begin{proof}
We have 
$
\bar{\rho}(X^L,Y^L) = \bar{R}ic(X^L,\phi Y^L) = \bar{R}ic(X^L,(JY)^L),
$
and in virtue of the Proposition {\bf \ref{p:ric}.}, 
\begin{align*}
& \bar{R}ic(X^L,(JY)^L) = Ric(X,JY) -2g(X,JY) = \\
& \qquad \rho(X,Y)-2\omega(X,Y),
\end{align*}
clearly $\bar{\rho}(X^L,\xi)=\bar{\rho}(\xi,X^L)=0$, therefore 
$\bar{\rho}$ is skew-symmetric and closed, as both Ricci and K\"ahler forms 
are closed.
\end{proof}

Here are some corollaries of obtained results.
\begin{theorem}
\label{th:chc}
If K\"ahler base has constant holomorphic sectional curvature $c=const.$, then 
its Sasakian lift is Sasakian manifold of constant $\phi$-sectional curvature 
$\bar{c} = c-3$.
\end{theorem}
\begin{proof}
Let fix a point $(t,q)\in \mathcal{M}$, and let $v\in \mathcal{D}_{(t,q)}$, be 
unit vector, then $\phi$-sectional curvature $K_\phi(v)$, is a sectional 
curvature of plane $(v,\phi v)$. Hence 
\begin{equation*}
K_\phi(v) = \bar{R}(v,\phi v,\phi v,v) =  g^L(\bar{R}_{v\phi v}\phi v, v).
\end{equation*} 
As $(\cdot)^L$ is point-wise linear isometry between $T_q\mathcal{N}$ and 
$\mathcal{D}_{(t,q)}$, there is local vector field $X$ on K\"ahler 
base, such that $X^L= v$ at the point $(t,q)$. We can assume that $X$ is 
normalized. In view of the Proposition {\bf \ref{p:curv}}, eq. (\ref{l:curv1}),
having in mind that $g^L(X^L,X^L)=g(X,X)$, and 
\begin{equation*}
g^L((R_{XJX}JX)^L,X^L) = g(R_{XJX}JX,X), 
\end{equation*}
we obtain
\begin{equation*}
\bar{R}(X^L,\phi X^L,\phi X^L,X^L) = R(X,JX,JX,X)-3g^2(X,X)= c-3,
\end{equation*}
where $c=R(X,JX,JX,X)$ is holomorphic sectional curvature of K\"ahler base. By 
assumption $c=const$, in particular at the point $(t,q)$, 
\begin{equation*}
\bar{R}(X^L,\phi X^L,\phi X^L,X^L) = K_\phi(v)= c-3.
\end{equation*}
As point and vector are arbitrary this shows that $\mathcal M$ has constant $\phi$-sectional curvature $c-3$.
\end{proof}

\begin{theorem}
\label{th:ein}
If K\"ahler base is K\"ahler-Einstein manifold with Einstein constant $c=const.$, then
its Sasakian lift is $\eta$-Einstein manifold 
\begin{equation}
\bar{R}ic = (c-2) \bar{g}+(2n-c+2)\eta\otimes\eta.
\end{equation} 
In particular is Einstein if and only if $c=2n+2$.
\end{theorem}

The particular case is when base K\"aler manifold has constant holomorphic 
curvature $c=4$. 
\begin{theorem}
If K\"ahler base is locally isometric to complex projective space 
$\mathbb{C}P^n$, equipped with Fubini-Study metric of constant holomorphic 
curvature $c=4$, then its lift is locally isometric to unit sphere 
$\mathbb{S}^{2n+1}\subset\mathbb{C}^{n+1}$, equipped with its canonical Sasakian 
structure of constant sectional curvature $\bar{c}=1$.
\end{theorem}
\begin{proof}
By Theorem {\bf \ref{th:chc}}, lift of the K\"ahler base of constant 
holomorphic curvature is Sasakian manifold of constant $\phi$-sectional 
curvature $\bar{c}=c-3$. Note that curvature operator of manifold with constant 
$\phi$-sectional curvature is completely determined, cf. {\bf Theorem  7.19, p. 139} in 
(\cite{Blair}). 
\end{proof}

Note that the case of dimension three is exceptional. As every two-dimensional K\"ahler manifold is Einstein, its lift is $\eta$-Einstein, yet coefficient $c$ now is 
in general a some function - in fact determined by Gaussian curvature of 2-dimensional 
base. 

\section{$\mathcal D_{\alpha,\beta}$-homothety of Sasakian manifolds and 
$\alpha$-Sasakian manifolds}
In this section we provide detailed study of $\mathcal D_{\alpha,\beta}$-homothety 
Sasakian manifold. Fundamental result here is that image of Sasakian 
manifold by some $\mathcal D_{\alpha,\beta}$-homothety with parameters $\alpha$, 
$\beta > 0$ is $\frac{\beta}{\alpha}$-Sasakian manifold.  

Let $\mathcal{M}$, be a Sasakian manifold with almost 
contact metric structure $(\phi,\xi,\eta,g)$. For real 
positive parameters $\alpha$, $\beta$, let consider 
new structure on $\mathcal{M}$, $(\phi'=\phi,\xi',\eta',g')$, 
given by
\begin{align}
& \xi' = \frac{1}{\beta}\xi, \quad
\eta'=\beta\eta, \\
& g' = \alpha g +(\beta^2-\alpha)\eta\otimes\eta.
\end{align}
It is useful to have explicitly inverse map for metric 
\begin{align}
 g = \frac{1}{\alpha}g'+(\frac{1}{\beta^2}-
 \frac{1}{\alpha})\eta'\otimes\eta'.
\end{align}

Here our main goal is to study relations between Levi-Civita 
connection $\nabla = LC(g)$, $\nabla' = LC(g')$, Riemann 
curvatures and in particular Ricci tensors. Note 
that in general deformed structure is not contact metric. 
It satisfies weaker condition
\begin{equation}
d\eta' = \frac{\beta}{\alpha}\Phi',
\end{equation}
where $\Phi'(X,Y)=g'(X,\phi Y) = \alpha \Phi(X,Y)$. 

\begin{proposition}
Let $\nabla$ be Levi-Civita connection of Sasakian manifold, 
and $\nabla'$ Levi-Civita connection of metric obtained 
by $\mathcal D_{\alpha,\beta}$-homotheties of Sasakian metric. 
Connections are related by the following formula
\begin{equation}
\label{e:txy}
\nabla_XY= \nabla'_XY+\frac{\beta^2-\alpha}{\alpha\beta}(\eta'(X)\phi Y +
\eta'(Y)\phi X). 
\end{equation} 
\end{proposition}
\begin{proof}
The proof is rather standard with help of Koszul formula for 
Levi-Civita connection. Let denote by $T_XY$, the difference 
tensor $\nabla_XY=\nabla'_XY+T_XY$. As connections are torsion-less
$T_XY$ is symmetric $T_XY=T_YX$. Therefore 
\begin{align*}
& -2g'(T_XY,Z)= 2g'(\nabla'_XY,Z) - 2g'(\nabla_XY,Z) = 
(\beta^2-\alpha)(X \eta(Y)\eta(Z)+ \\
& \qquad Y\eta(X)\eta(Z) - Z\eta(X)\eta(Y)) + 
(\beta^2-\alpha)(\eta([X,Y])\eta(Z) + \\
& \qquad \eta([Z,X])\eta(Y) + \eta([Z,Y])\eta(X)) - 
2(\beta^2-\alpha)\eta(\nabla_XY)\eta(Z),
\end{align*}
note $X\eta(Y)\eta(Z) = (X\eta(Y))\eta(Z)+\eta(Y)(X\eta(Z))$, just 
well-known Leibniz rule, moreover 
$\eta([X,Y])= \eta(\nabla_XY)-\eta(\nabla_YX)$. Therefore using Leibniz 
rule after regrouping we find
\begin{align*}
& (\beta^2-\alpha)(X\eta(Z)-Z\eta(X)+\eta([Z,X]))\eta(Y) + \\
& \qquad (\beta^2-\alpha)(Y\eta(Z)-Z\eta(Y)+\eta([Z,Y]))\eta(X) + \\
& \qquad (\beta^2-\alpha)(X\eta(Y)-\eta(\nabla_XY))\eta(Z)) + \\
& \qquad (\beta^2-\alpha)(Y\eta(X)-\eta(\nabla_YX))\eta(Z)) = \\
& \qquad 2(\beta^2-\alpha)(d\eta(X,Z)\eta(Y)+d\eta(Y,Z)\eta(X)) + \\
& \qquad (\beta^2-\alpha)((\nabla_X\eta)(Y)+(\nabla_Y\eta)(X))\eta(Z),
\end{align*}
however on Sasakian manifold $(\nabla_X\eta)(Y)+(\nabla_Y\eta)(X)=0$, 
thus  we obtain
\begin{equation*}
-g'(T_XY,Z)= (\beta^2-\alpha)(d\eta(X,Z)\eta(Y)+d\eta(Y,Z)\eta(X)).
\end{equation*}
Note $d\eta=\frac{1}{\beta}d\eta' = \frac{1}{\alpha}\Phi'$, in terms 
of deformed structure the above equation reads
\begin{equation*}
-g'(T_XY,Z)= \frac{\beta^2-\alpha}{\alpha\beta}( \Phi'(X,Z))\eta'(Y)+
\Phi'(Y,Z)\eta'(X)),
\end{equation*}
from $\Phi'(X,Y)=g'(X,\phi Y)= -g'(\phi X,Y)$, we finally obtain
\begin{equation}
T_XY = \frac{\beta^2-\alpha}{\alpha\beta}(\eta'(X)\phi Y +\eta'(Y)\phi X),
\end{equation}
ie. (\ref{e:txy}).
\end{proof}
Of course once we have explicit form of tensor $T_XY$ we can directly verify that $\nabla'g'=0$, and 
use the fact that Levi-Civita connection is unique connection which is both torsion-less and 
$\nabla'g'$.  Such direct verification provides alternative proof of the above statement.  

For further reference we set 
$c=c_{\alpha,\beta}=\frac{\beta^2-\alpha}{\alpha\beta}$.
\begin{proposition}
Covariant derivative $\nabla'\phi$ is 
given by
\begin{equation}
\label{e:nabpfi}
(\nabla'_X\phi)Y = \frac{\beta}{\alpha}(g'(X,Y)\xi'-\eta'(Y)X). 
\end{equation}
\end{proposition}
\begin{proof}
 
We have
\begin{align}
\label{e:nab1}
& (\nabla_X\phi )Y = (\nabla'_X\phi)Y +(T_X\phi)Y= (\nabla'_X\phi)Y + \\
& \nonumber\qquad T_X\phi Y - \phi T_XY = (\nabla'_X\phi)Y + 
c(\eta'(X)\phi^2Y - \\ 
& \nonumber\qquad \eta'(X)\phi^2Y-\eta'(Y)\phi^2X) = (\nabla'_X\phi)Y + 
c(\eta'(Y)X-\eta'(X)\eta'(Y)\xi'),
\end{align}
As $\mathcal{M}$ is Sasakian 
\begin{align}
\label{e:nab2}
& (\nabla_X\phi)Y=g(X,Y)\xi -\eta(Y)X = \frac{\beta}{\alpha}g'(X,Y)\xi' - \\
&\nonumber \qquad c\eta'(X)\eta'(Y)\xi' - 
  \frac{1}{\beta}\eta'(Y)X,
\end{align}
comparing (\ref{e:nab1}), (\ref{e:nab2}), we obtain (\ref{e:nabpfi}).
\end{proof}

\begin{corollary}
Image of Sasakian manifold by $\mathcal D_{\alpha,\beta}$-homothety with parameters 
$\alpha$, $\beta > 0$ is $\frac{\beta}{\alpha}$-Sasakian manifold. 

\end{corollary}

To find relation between corresponding curvature operators we use following 
general formula
\begin{equation}
\label{e:r:rp}
R_{XY}Z=R'_{XY} +(\nabla'_XT)_YZ-(\nabla'_YT)_XZ+[T_X,T_Y]Z,
\end{equation}
where $[T_X,T_Y]Z=T_XT_YZ-T_YT_XZ$.

For covariant derivative $\nabla'_XT$ on base of above Propositions we find
\begin{align}
\label{e:nabptxy}
& (\nabla'_XT)_YZ =c ((\nabla'_X\eta')(Y)\phi Z + (\nabla'_X\eta')(Z)\phi Y +\\
& \nonumber\qquad \eta'(Y)(\nabla'_X\phi)Z+\eta'(Z)(\nabla'_X\phi)Y) =  \\
 & \nonumber\qquad \frac{c\beta}{\alpha}(\Phi'(X,Y)\phi Z+ \Phi'(X,Z)\phi Y )+ \\ 
& \nonumber\qquad \frac{c\beta}{\alpha}(g'(X,Z)\eta'(Y)+g'(X,Y)\eta'(Z))\xi' - \\ 
& \nonumber\qquad \frac{c\beta}{\alpha} 2\eta'(Y)\eta'(Z)X),
\end{align}
we have used $(\nabla_X\eta')(Y)=\frac{\beta}{\alpha}\Phi'(X,Y)$. 
For $T_XT_YZ$, we obtain
\begin{align}
\label{e:txtyz}
& T_XT_YZ = c\eta'(X)\phi T_YZ = c^2(\eta'(X)\eta'(Y)\phi Z + \\
& \nonumber\qquad\eta'(X)\eta'(Z)\phi Y)
\end{align}
as $\eta'(T_XY)=0$, for every $X$, $Y$. In view of 
(\ref{e:r:rp}),(\ref{e:nabptxy}),(\ref{e:txtyz}), we can establish 
following result.
\begin{proposition}
Let $g'$ be a $\mathcal D_{\alpha,\beta}$-homothety of Sasakian metric. Then 
Riemann curvature operator $R$ and the curvature operator $R'$ of the
deformed metric are related by following formula
\begin{align}
& R_{XY}Z = R'_{XY}Z + \\
& \nonumber\qquad  \frac{c\beta}{\alpha}( \Phi'(X,Z)\phi Y -\Phi'(Y,Z)\phi X+2\Phi'(X,Y)\phi Z) + \\
& \nonumber\qquad  \frac{c\beta}{\alpha}(g'(X,Z)\eta'(Y)-g'(Y,Z)\eta'(X))\xi' - \\
& \nonumber \qquad \frac{c\beta}{\alpha}\eta'(Z)(\eta'(Y)X-\eta'(X)Y)) - \\
& \nonumber \qquad  \frac{c^3\beta}{\alpha}\eta'(Z)(\eta'(Y)\phi X-\eta'(X)\phi Y), 
\end{align} 
\end{proposition}

We have $Ric(Y,Z) = Tr \{ X\mapsto R_{XY}Z\}$, so as corollary from the above proposition we obtain
\begin{corollary}
Ricci tensors of Sasakian manifold and its $\frac{\beta}{\alpha}$-Sasakian deformation  are related by 
\begin{equation}
Ric(Y,Z)=Ric'(Y,Z)+2\frac{c\beta}{\alpha}(g'(Y,Z)-(n+1)\eta'(Y)\eta'(Z)).
\end{equation}
\end{corollary}

As we know already $\mathcal D_{\alpha,\beta}$-homothety of Sasakian manifold 
is an $\frac{\beta}{\alpha}$-Sasakian manifold. Providing two consecutive 
homotheties with parameters $(\alpha_i,\beta_i)$, $i=1,2$, we obtain that 
resulting manifold is $\frac{\beta_1\beta_2}{\alpha_1\alpha_2}$-Sasakian. 
Therefore as conclusion we obtain general statement that 
$\mathcal D_{\alpha,\beta}$-homothety with parameters $\alpha_1$, $\beta_1$ of 
some $\alpha$-Sasakian manifold is 
$(\frac{\beta_1}{\alpha_1}\alpha)$-Sasakian.

\subsection{$\mathcal D_{\alpha,\beta}$-homotheties of $\alpha$-Sasakian twisted $\eta$-Ricci soliton}
In this part of the paper we will prove important result that 
equation which defines twisted $\eta$-Ricci soliton is on 
$\alpha$-Sasakian manifolds, invariant under $\mathcal D_{\alpha,\beta}$-homotheties. 

\begin{theorem}
Let assume $\alpha$-Sasakian manifold $(\mathcal{M},\phi,\xi,\eta,g)$ 
is twisted $\eta$-Ricci soliton
\begin{align}
\label{e:sas:esol}
& Ric +\frac{1}{2}(\mathcal{L}_Xg) = \lambda g+2C_1 \alpha_X\odot\eta +
C_2\eta\otimes\eta, \\
& (\mathcal{L}_X\eta)(Y) = \alpha_X(Y),\quad
\eta(X)=0, 
\end{align}
then its image by $\mathcal D_{\alpha,\beta}$-homothety $(\mathcal{M},\phi,\xi',\eta',g')$,
 is also a twisted $\eta$-Ricci soliton 
soliton, 
\begin{align}
& Ric' +\frac{1}{2}\mathcal{L}_{X'}g' = \lambda'g' + 
2C_1'\alpha'_{X'}\odot \eta' +C_2'\eta'\otimes\eta', \\
& (\mathcal{L}_{X'}\eta')(Y) = \alpha'_{X'}, \quad 
 \quad \eta'(X')=0,
\end{align} 
\end{theorem}
\begin{proof}
The proof is almost evident. Without loosing generality we may assume that 
manifold is Sasakian. We only need to find how each term in equation 
(\ref{e:sas:esol}) changes under $\mathcal D_{\alpha,\beta}$-homothety. Therefore
\begin{align}
& Ric = Ric'+2\frac{c\beta}{\alpha}g'-
  2(n+1)\frac{c\beta}{\alpha}\eta'\otimes \eta', \\
& \mathcal{L}_Xg= \frac{1}{\alpha}\mathcal{L}_Xg' + 
2(\frac{1}{\beta^2}-\frac{1}{\alpha})(\mathcal{L}_X\eta')\odot\eta' = \\
& \nonumber\qquad\mathcal{L}_{X'}g'+
2(\frac{\alpha}{\beta^2}-1)(\mathcal{L}_{X'}\eta')\odot\eta', \\
& \alpha_X = \mathcal{L}_X\eta = 
\frac{\alpha}{\beta}\mathcal{L}_{X'}\eta' = 
\frac{\alpha}{\beta}\alpha'_{X'}, \\
& g = \frac{1}{\alpha}g'+(\frac{1}{\beta^2}-\frac{1}{\alpha})\eta'\otimes\eta', \quad
\eta\otimes\eta = \frac{1}{\beta^2}\eta'\otimes\eta',
\end{align}
however we need to rescale vector field $X$ by $\frac{1}{\alpha}$,
$X \mapsto X'=\frac{1}{\alpha}X$, 
after regrouping we obtain that $\frac{\beta}{\alpha}$-Sasakian manifold 
is twisted $\eta$-Ricci soliton with constants $\lambda'$, 
$C_1'$, $C_2'$, given by 
\begin{align*}
& \lambda' = \frac{1}{\alpha}(\lambda-2c\beta), \quad 
C_1' = \frac{\alpha}{\beta^2}(C_1-\frac{1}{2})+\frac{1}{2}, \\ 
& C_2' = \lambda(\frac{1}{\beta^2}-\frac{1}{\alpha}) + 
\frac{C_2}{\beta^2}+2(n+1)c\frac{\beta}{\alpha}
\end{align*}
where $c=\frac{\beta^2-\alpha}{\alpha\beta}$.
\end{proof}

On the base of above formulas we can answer question whether or not it is possible to remove the 
twist from the equation. That means does exist $D_{\alpha,\beta}$-homothety so $C_1=0$? 
From above formulas we see that necessary and sufficient condition is that the source structure 
satisfies $C_1 < \frac{1}{2}$.  Behavior under deformations determines three classes of 
$\alpha$-Sasakian twisted $\eta$-Ricci solitons determined by value of twist coefficient $C_1$. 
The first class are those manifolds where $C_1 < \frac{1}{2}$, second say singular class are manifolds 
where $C_1=\frac{1}{2}$, and the third class are manifolds where $C_1 > \frac{1}{2}$. Later on 
studying lifts of Ricci-K\"ahler solitons as corollary we obtain that class $C_1 < \frac{1}{2}$ is 
always nonempty. Exactly lift of Ricci-K\"ahler soliton belongs to this class. So basically there is 
problem to solve about the other two classes: do exist $\alpha$-Sasakian twisted 
$\eta$-Ricci solitons with $C_1 \geqslant \frac{1}{2}$?  Note that it is the case where it is not possible 
to remove twist by $\mathcal D_{\alpha,\beta}$-homothety.

\section{Lifts of Killing vector fields, inifinitesimal biholomorhpisms and
automorhpisms}

In this we are interested in lifts of vector fields which satisfy some 
additional conditions. We just want to ask questions what is 
a lift of complex structure infinitesimal automorphism and 
similarly what is a lift of Killing vector field. 
\begin{proposition}
\label{p:inaut}
Let $V$ be an infinitesimal automorphism of complex structure on K\"ahler 
base. Then its lift $V^L$ satisfies 
\begin{equation}
(\mathcal{L}_{V^L}\phi)X^L = 2g^L(V^L,X^L)\xi, 
\quad (\mathcal{L}_{V^L}\phi)\xi = 0.
\end{equation}
\end{proposition}
\begin{proof}
We have 
\begin{equation}
\label{e:v:phi}
[V^L,\phi X^L] = [V, JX]^L -2 \Phi(V^L,\phi X^L)\xi, 
\end{equation}
\begin{equation}
\label{e:phi:v}
\phi[V^L,X^L] = \phi [V,X]^L = (J[V,X])^L,
\end{equation}
as $\Phi(V^L,\phi X^L) = -g^L(V^L,X^L)$, by (\ref{e:v:phi}), 
 (\ref{e:phi:v}) 
\begin{align*}
& (\mathcal{L}_{V^L}\phi)X^L = [V^L,\phi X^L] -\phi[V^L,X^L]= 
((\mathcal{L}_VJ)X)^L+
\\ 
& \nonumber \qquad 2g^L(V^L,X^L)\xi,
\end{align*}  
and result follows by assumption that $\mathcal{L}_VJ=0$.
\end{proof}

In similar way we can prove following statement considering lift 
of Killing vector fields from K\"ahler base.
\begin{proposition}
\label{p:kill}
Let $V$ be a Killing vector field on K\"ahler base. Then its lift 
satisfies
\begin{equation}
(\mathcal{L}_{V^L}g^L)(X^L,Y^L) = 0, \quad
(\mathcal{L}_{V^L}g^L)(\xi, X^L) = 2\Phi(V^L,X^L).
\end{equation}
\end{proposition}

\begin{proposition}
Let $V$ be an inifinitesimal automorphism of K\"ahler form. Its lift 
satisfies
\begin{equation}
\label{e:lie:vl:phi}
\mathcal{L}_{V^L}\Phi = 0,
\end{equation}
in particular one-form $\mathcal{L}_{V^L}\eta$ is closed.
\end{proposition}
\begin{proof}
To proof (\ref{e:lie:vl:phi}), we show that 
$(\mathcal{L}_{V^L}\Phi)(X^L,Y^L)=0$, and 
$(\mathcal{L}_{V^L}\Phi)(X^L,\xi)=0$. As Lie derivative and exterior 
derivative commute we have
\begin{equation}
0=\mathcal{L}_{V^L}\Phi = \mathcal{L}_{V^L}d\eta = 
d(\mathcal{L}_{V^L}\eta),
\end{equation}
hence $\mathcal{L}_{V^L}\eta$ is closed one-form.
\end{proof}

Now we provide result establishing kind of relationship between 
local symmetries of K\"ahler base and some vector fields on its
 Sasakian lift
\begin{theorem}
Let $V$ be inifinitesimal automorphism of K\"ahler structure of K\"ahler 
base. Its lift $V^L$, satisfies 
\begin{equation}
\mathcal{L}_{V^L}\phi = 2\alpha\otimes\xi , \quad 
\mathcal{L}_{V^L}g^L = 4\alpha^\phi\odot\eta, \quad
\mathcal{L}_{V^L}\Phi = 0,
\end{equation}
where $\alpha(\cdot)=g^L(V^L,\cdot)$, 
$\alpha^\phi(\cdot)=\alpha(\phi\cdot)=-g^L(\phi V^L,\cdot)$. and 
form $\alpha^\phi$ is closed, $d\alpha^\phi=0$.
\end{theorem}
\begin{proof}
Only what requires is a proof that form $\alpha^\phi$ is closed. 
Hover note that we may identify $\alpha^\phi$, with a pullback 
of closed form $X\mapsto \omega(V,X)$.
\end{proof}

Here we formulate main result of this section.
\begin{theorem}
Let $V$ be automorphism of K\"ahler structure of K\"ahler base. Then 
there is locally defined function $f$, $df(\xi)=0$, on Sasakian lift, such that
vector field $U_V=V^L+f\xi$, is local infinitesimal automorphism of 
almost contact metric structure of Sasakian lift.
\end{theorem}
\begin{proof}
From the properties of the Lie derivative
\begin{equation*}
(\mathcal{L}_{U_V}\phi)Y^L=(\mathcal{L}_{V^L}\phi)Y^L + 
f(\mathcal{L}_\xi\phi)Y^L-df(\phi Y^L)\xi,
\end{equation*}
on Sasakian manifold $\mathcal{L}_\xi\phi=0$, in the view of the Proposition {\bf\ref{p:inaut}}., we obtain
\begin{equation*}
(\mathcal{L}_{U_V}\phi)Y^L = (2g^L(V^L,Y^L)-df(\phi Y^L))\xi.
\end{equation*}
As $V$ is an automorphism of K\"ahler structure, is locally Hamiltonian, with 
respect to K\"ahler form $\omega(V,Y)=dH(Y)$, for locally defined function $H$.
Set $f=-2\bar{H}=-2H\circ\pi_2$, then  
\begin{equation*}
df(\phi Y^L) = df((JY)^L) = -2dH(JY)= -2\omega(V,JY)=2g(V,Y).
\end{equation*}
 Having 
$f$ determined  we verify 
directly that 
\begin{equation*}
(\mathcal{L}_{U_V}\phi)\xi = 0, \quad (\mathcal{L}_{U_V}g^L)(Y^L,Z^L)=0, 
\quad (\mathcal{L}_{U_V}g^L)(\xi,Y^L)=0.
\end{equation*}
\end{proof}

\section{Sasakian lift of K\"ahler-Ricci soliton and 
$\alpha$-Sasakian $\eta$-Ricci solitons }
In this section we are interested particularly in lifts of 
K\"ahler-Ricci solitons. 
Let vector field $X$ satisfies Ricci soliton equation  
on K\"ahler base
\begin{equation}
\label{e:ric:sol}
Ric+\dfrac{1}{2}\mathcal{L}_Xg = \lambda g,
\end{equation}
where $\lambda = const.$ is  a real constant. 
By direct computations we find 
$(\mathcal{L}_{X^L}g^L)(Y^L,Z^L) = (\mathcal{L}_Xg)(Y,Z)$, 
then by (\ref{e:ric:sol})
\begin{equation*}
\frac{1}{2}(\mathcal{L}_{X^L}g^L)(Y^L,Z^L) = 
\frac{1}{2}(\mathcal{L}_Xg)(Y,Z)= \lambda g(Y,Z)-Ric(Y,Z),
\end{equation*}
From other hand by the Proposition {\bf\ref{p:ric}.}, 
$$ 
\bar{R}ic(Y^L,Z^L)=Ric(Y,Z)-2g(Y,Z),
$$ 
summing up we find
\begin{align}
\label{e:ric:sol2}
& \bar{R}ic(Y^L,Z^L)+\frac{1}{2}(\mathcal{L}_{X^L}g^L)(Y^L,Z^L) = 
Ric(Y,Z)-2g(Y,Z) + \\
& \nonumber \lambda g(Y,Z)-Ric(Y,Z) = (\lambda-2)g^L(Y^L,Z^L),
\end{align}
in similar way
\begin{equation}
\label{e:ric:sol3}
\bar{R}ic(\xi,Y^L)+\frac{1}{2}(\mathcal{L}_{X^L}g^L)(\xi,Y^L)= \Phi(X^L,Y^L),
\end{equation}
\begin{equation}
\label{e:ric:sol4}
\bar{R}ic(\xi,\xi)+\frac{1}{2}(\mathcal{L}_{X^L}g^L)(\xi,\xi) = 2n.
\end{equation}
The identities (\ref{e:ric:sol2})-(\ref{e:ric:sol4}), allow us to state the following result
\begin{theorem}
Let  K\"ahler base be a K\"ahler-Ricci soliton.  Sasakian lift is twisted $\eta$-Ricci soliton
\begin{align}
& \bar{R}ic+\frac{1}{2}\mathcal{L}_{X^L}g^L= (\lambda-2)g^L-2(\mathcal L_{X^L}\eta)\odot\eta + (2n-2+\lambda)\eta\otimes \eta, 
\end{align}
\end{theorem}

\begin{corollary}
Sasakian lift of K\"ahler-Ricci soliton admits $\mathcal D_{\alpha,\beta}$-homothety 
so $\alpha$-Sasakian image is $\eta$-Ricci soliton.
\end{corollary}
\begin{proof}
Sasakian lift satisfies equation of twisted $\eta$-Ricci soliton. By 
{\bf Corollary \ref{c:tesol:esol} .} there is a $\mathcal D_{\alpha,\beta}$-homothety so image is $\alpha$-Sasakina  $\eta$-Ricci soliton.
\end{proof}

As Ricci tensor determines a 2-form on Sasakian lift above equation can be expressed in the following form
\begin{proposition}
 Sasakian lift of K\"ahler-Ricci soliton, satisfies 
following equations with resp. to the fundamental form and Ricci form of Sasakian 
lift
\begin{equation}
\bar{\rho}+\frac{1}{2}\mathcal{L}_{X^L}\Phi = (\lambda-2)\Phi.
\end{equation}
\end{proposition}

Above corollary give us plenty of examples of $\eta$-Ricci solitons.

\end{document}